\documentclass[a4paper,10pt]{amsart}
\usepackage{palatino, mathpazo}
\usepackage[plainpages=false,colorlinks=true,linkcolor=blue,citecolor=blue,urlcolor=blue]{hyperref}
\usepackage{amsfonts,latexsym,rawfonts,amsmath,amssymb,amsthm, mathrsfs, lscape}
\usepackage{verbatim}
\usepackage{stmaryrd}
\usepackage[T1]{fontenc}
\usepackage[utf8]{inputenc}
\usepackage{amsmath,amssymb,amsthm,mathtools,array,booktabs}
\usepackage[margin=3.6cm]{geometry}
\usepackage[colorlinks=true,linkcolor=blue,citecolor=blue,urlcolor=blue]{hyperref}

\newtheorem{theorem}{Theorem}[section]
\newtheorem{proposition}[theorem]{Proposition}
\newtheorem{corollary}[theorem]{Corollary}
\newtheorem{lemma}[theorem]{Lemma}
\newtheorem{remark}[theorem]{Remark}
\newcommand{\C}{\mathbb C}
\newcommand{\Q}{\mathbb Q}
\newcommand{\A}{\mathbb A}
\newcommand{\PP}{\mathbb P}
\newcommand{\Hilb}{\operatorname{Hilb}}
\newcommand{\Crit}{\operatorname{Crit}}
\newcommand{\Tr}{\operatorname{Tr}}

\begin{document}

\title{Local monodromy, Behrend function and  Hilbert scheme of points on $\C^3$}

\author{Yunfeng Jiang}
\address{Department of Mathematics\\ University of Kansas\\ 405 Snow Hall 1460 Jayhawk Blvd\\Lawrence KS 66045 USA} 
\email{y.jiang@ku.edu}
\date{}
\maketitle

\begin{abstract}
We prove  a criteria for the Behrend function values on the Hilbert scheme of $n$-points on $\C^3$ from local monodromy
and the Denef--Loeser's  Lefschetz trace formula of jet spaces.
If the power two, three, or four of the monodromy  is unipotent, it reduces respectively to a Hessian
parity formula, a cubic-complement formula, or a quartic
formula.  This gives a Milnor fiber and monodromy based proof of nonconstancy of the Behrend function 
on $\Hilb^n(\C^3)$ after Jelisiejew-Kool-Schmiermann.  We also calculate two ideals in $\Hilb^9(\C^3)$ and $\Hilb^{12}(\C^3)$ so that the Behrend function is not $\pm 1$, answering a question of Ricolfi. 
\end{abstract}

\section{Introduction}

We work over $\C$ in this paper.  The Donaldson-Thomas invariants of Calabi-Yau threefolds \cite{Thomas}, \cite{B} provide new phenomenon on the classical moduli  schemes. For instance, the local charts of the Donaldson-Thomas moduli space of stable sheaves on a Calabi-Yau threefold $Y$ are critical locus of holomorphic functions.
This gives the Donaldson-Thomas moduli space a $(-1)$-shifted symplectic scheme structure.  If $Y=\C^3$, the simplest Calabi-Yau threefold,  then the degree zero  Donaldson-Thomas moduli space is the Hilbert scheme $X_n:=\Hilb^n(\C^3)$ of points on $\C^3$, which is a global critical locus. 

The degree zero Donaldson-Thomas invariants of $\C^3$ is given by the weighted Euler characteristic of $X_n$, weighted by the Behrend function on $X_n$. The generating function of the degree zero Donaldson-Thomas invariants is the MacMahon function, see \cite{MNOP2}, \cite{Li},  \cite{B}, \cite{BF}, \cite{LP}.  

The Behrend function $\nu_X: X\to \mathbb Z$ for a scheme $X$ is an integer valued constructible function, which plays an important role in Donaldson-Thomas theory for Calabi-Yau threefolds.  Behrend's celebrated theorem \cite{B} states that the Donaldson-Thomas virtual count is the $\nu_X$-weighted Euler characteristic of $X$ if $X$ is the moduli space of stable sheaves on a Calabi-Yau threefold $Y$.

Behrend-Fantechi's method \cite{BF} for calculating the weighted Euler characteristic of $X_n$ is to calculate the Behrend function on the $\C^*$-fixed points on $\Hilb^n(\C^3)$ (which are given by monomial ideals).  The Behrend function is $(-1)^n$ for each $\C^*$-fixed ideal. Thus, it was conjectured in the community that the Behrend function $\nu_{X_n}$ is a global constant $(-1)^n$.
But from the complexity of the Hilbert scheme $X_n$ when $n$ is large (which is not reduced, and not even irreducible, see \cite{Jelisiejew}), the conjecture is false and was first proved in a recent paper \cite{JKS}, where Jelisiejew-Kool-Schmiermann calculated that 
$$
\nu_{\Hilb^{24}(\C^3)}(I)=-1
$$
for an ideal $I\in \Hilb^{24}(\C^3)$ of length $24$.  
Jelisiejew-Kool-Schmiermann used twice of the result of  equivariant symmetric obstruction theory as in \cite{B}, \cite{BF} to calculate the above value.  The theorem in \cite{BF} says that if a scheme $X$ admits a $\C^*$-equivariant symmetric obstruction theory  and if $P$ is an isolated fixed point, then 
$$
\nu_{X}(P)=(-1)^{\dim(T_P(X))}.
$$

Recall from  \cite{Sze08}, \cite{Ric22}, \cite{Ric24},  the scheme  $X_n$ is a global critical locus $\Crit(f_n)$, where $f_n: M\to\C$ is a trace function on a smooth manifold $M$ of dimension $N=2n^2+n$.  Here $$M=\widetilde{M}/\mathrm{GL}_n,$$ 
where $\widetilde M\subset \mathrm{End}(\C^n)^3\times\C^n$ is the open subset of cyclic commuting quadruples. The potential function is
\[
 f_n=\Tr(A[B,C]).
\]

Then for any $P\in X_n$, Behrend \cite{B} shows that
$$
\nu_{X_n}(P)=(-1)^{N}(1-\chi(F_P))
$$
where $\chi(F_P)$ is the Euler characteristic of the Milnor fiber of $f_n$ at $P\in X_n$. Thus, the Behrend function is not globally constant if and only if the Euler characteristic $\chi(F_P)$ is not globally zero. 
In this paper we give a proof of the nonconstancy of the Behrend function $\nu_{X_n}$ from the results of the Euler characteristic of the  Milnor fiber without using any $\C^*$-equivariant symmetric obstruction theory.  We state the main result. 

\subsection{Main result}
From \cite{Loeser}, the Milnor fibration of $f_n$ 
determines  an automorphism on the Milnor fiber $F_{P}$, defined
up to homotopy and called the local monodromy $T_P$ at $P$. In particular the singular
cohomology groups $H^i(F_P,\Q)$ are endowed with an automorphism $T_P$.
We are interested in the case that 
\begin{equation}\label{eq:nilpotence}
 (T_P^m-I)^q=0\quad\text{on every }H^j(F_P,\Q)
 \quad\text{for some }q\ge1.
\end{equation}
We call $T_P^m$ is \emph{unipotent} if (\ref{eq:nilpotence}) satisfies. 
Put $E=\operatorname{End}(\C^n)^3$ and write $z=(a,b,c)\in E$.
Define
\begin{align*}
 Q(z)=\Tr\bigl(A_0[b,c]+a[B_0,c]+a[b,C_0]\bigr),\quad \quad
 C(z)=\Tr(a[b,c]).
\end{align*}
Then the exact expansion is 
\begin{equation}\label{eq:exacttrace}
   \mathrm{Tr}(A_0+a)[B_0+b, C_0+c]=Q(z)+C(z).
\end{equation}

Let $B$ be the polarization of $Q$, with $Q(z)=B(z,z)$,
and set
\[
 r=\operatorname{rank} B=N-\dim T_P X_n,
 \qquad K=\ker B,\qquad c=C|_{K}.
\]
Here $K$ is the kernel in the \emph{matrix space} $E$. 
Choose a complement $L$ to $K$ and define the invertible map
\[
 \mathsf B:L\longrightarrow L^*,\qquad
 \mathsf B(u)=B(u,-).
\]
For $z\in K$, let
\begin{align*}
 \ell_z(w)&=(dC)_z(w)\quad(w\in L),\\
 h(z)&=-\frac14\ell_z(\mathsf B^{-1}\ell_z),
       \label{eq:tracequartic}\\
 \Sigma&=\{z\in K:(dC)_z=0\}.
\end{align*}
The derivative in this definition
is explicitly
\[
 (dC)_{(a,b,c)}(\alpha,\beta,\gamma)
 =\Tr\bigl(\alpha[b,c]+a[\beta,c]+a[b,\gamma]\bigr).
\]
Thus all inputs are obtained by finite linear algebra and polynomial
operations on the commuting triple. Define the integers
\[
 b_{3}=\chi_c\bigl(\PP(K)\setminus V(c)\bigr),\qquad
 b_{4}=\chi_c\bigl(\PP(\Sigma)\setminus V(h)\bigr).
\]
We use underlying reduced spaces and
$\PP(\Sigma)=(\Sigma\setminus\{0\})/\C^*$, and  empty spaces have
Euler characteristic zero.

\begin{theorem}\label{thm:intromain}
The following statements hold at $P\in X_n$.
\begin{enumerate}
\item If $T_P^2$ is unipotent, then
\[
 \chi(F_P)=1-(-1)^{r}
 =\begin{cases}0,&r\text{ even},\\2,&r\text{ odd}.\end{cases}
\]
Equivalently $\nu_{X_n}(P)=(-1)^{\dim T_PX_n}$.
\item If $T_P^3$ is unipotent, then
\[
 \chi(F_P)=\chi_c\{z\in K:C(z)=1\}=3b_{3}.
\]
Thus $\nu_{X_n}(P)=(-1)^n(1-3b_{3})$.
\item If $T_P^4$ is unipotent, then
\[
 \chi(F_P)=1+(-1)^{r}(4b_{4}-1)
 =\begin{cases}4b_{4},&r\text{ even},\\
                2-4b_{4},&r\text{ odd}.\end{cases}
\]
Thus $\nu_{X_n}(P)=(-1)^{n+r}(1-4b_{4})$.
In particular $\chi(F_P)=0$ if and only if $r$ is even and
$b_{4}=0$.
\end{enumerate}
Without the corresponding unipotence hypothesis, the same right-hand
sides compute $L(T_P^2)$, $L(T_P^3)$, and $L(T_P^4)$ respectively.
\end{theorem}

The proof combines the matrix-lift argument  with the direct
jet computations in Theorem~\ref{thm:square}, Theorem~ \ref{thm:cube}, and
Theorem~ \ref{thm:fourth}. 

\begin{remark}
    The order of the unipotency of the local monodromy $T_P$ for $P\in X_n$ depends on the points $P$.  Thus, it is interesting to ask if there is a bound for the unipotency order for all $P\in X_n$ for a fixed $n$.
\end{remark}

To prove the global nonconstancy of the Behrend function $\nu_{X_n}$, we need to find ideals $P$ in $X_n$ such that 
$T_P^2$, $T_P^3$, or $T_P^4$ is unipotent.
The following ideal
\[
 I=\bigl((x^2)+(y,z)^2\bigr)^2+(y^3-x^3z)
\]
is homogeneous for weights $(1,1,0)$. Its trace potential has weight
$2$, so  $T_P^2=1$ at $[I]$.
The tangent dimension  is $99$, which is  calculated in
\cite[Theorem 6]{JKS},
\[
 N=1176,\qquad r=1176-99=1077.
\]
Theorem~\ref{thm:intromain} therefore yields $\chi(F_{P})=2$ and
$\nu_{X_{24}}([I])=-1$.  \cite[Theorem 6]{JKS} used equivariant symmetric obstruction theory of $(X_{24},P)$ to deduce the above formula. 

\begin{theorem}\label{thm:family}
Let $S=\C[x,y,z]$ and 
$\mathfrak m=(x,y,z)$. We have  the following results of families of ideals: 
\begin{enumerate}
    \item $I_m=((x^3)+(y,z^m)^2)^2+(y^3-x^3z^m)$ has length $36m$, and $T_P^2$ is unipotent;
    \item $B_d=\mathfrak m^{d+1}+(x^d+y^d+z^d)$  has length $\binom{d+3}{3}-1$, and $T_P^3$ is unipotent;
\item $D_s=((x^s)+(y^s,z^s)^2)^2+(y^{3s}-x^sz^s)$ has length  $12s^3$ and  $T_P^4$ is unipotent,
\end{enumerate}
where 
\begin{enumerate}
\item[(i)]
for $P=I_m$,  odd $m$ gives monodromy order two and $\chi(F_P)=2$, and 
even $m$ gives $T_P=1$. In particular, we have
\[
\dim_{\C}S/I_m=36m,
\qquad
\dim T_{[I_m]}X_{36m}=135m,
\]
and
\[
\nu_{H_{36m}}([I_m])=(-1)^m,
\qquad
\chi(F_{[I_m]})=1-(-1)^m.
\]
Every subscheme in this family is supported at
the origin. Thus for odd $m$, the results prove  the global nonconstancy of the Behrend function $\nu_{X_n}$ for $n=36m$.

\item[(ii)] for $P=B_d$, when $d=2$, the length $B_2$ is $9$, we have 
$$
\nu_{X_{9}}(B_2)=2.
$$

\item[(iii)] for $P=D_s$, 
$\chi(F_P)\equiv2\pmod4$ when $s$ is odd, which already implies that $\chi(F_P)$ is not zero.  Thus, the Behrend function $\nu_{X_n}$ is already nonconstant for $n=12s^3$.  In particular,
$$
\nu_{X_{12}}(D_1)=3.
$$
\end{enumerate}
\end{theorem}

\subsection{Convention}
For a Milnor fiber $F_P$, 
all $H^j(F_P)$  mean ordinary cohomology, including $H^0$.
All Euler characteristics of algebraic jet spaces are compactly
supported Euler characteristics, denoted $\chi_c$. For complex algebraic
varieties, the numerical values of ordinary and compactly supported
Euler characteristic agree.

\subsection{Outline}

In \S \ref{sec_trace_jet} we collect some known facts about the local monodromy and jet-spaces.  \S \ref{sec:local:invariants} proves some local invariants for the Hilberrt scheme $\Hilb^n(\C^3)$. In \S \ref{sec:proof_1.1} we prove Theorem \ref{thm:intromain}.
In \S \ref{sec:torus_action} we find the torus action weights so that the local monodromy of the  ideals is unipotent; and in \S \ref{sec:examples} we provide examples and prove Theorem \ref{thm:family}.  
In the appendix, we collect some interesting facts of local motivic and monodromy zeta functions and relate them to the weak monodromy conjecture.

\subsection*{Acknowledgement}

The nonconstancy of the Behrend function on the Hilbert scheme of points on $\C^3$ is a long time interesting problem.  I would like to thank Professors  Kai Behrend, Jim Bryan, and Richard Thomas for discussions when I was visiting UBC, and when I was in Imperial College. 
Y. J. thanks  Martijn Kool, Andrea Ricolfi, and Promit Kundu for valuable discussions. 

Although the idea to prove the nonconstancy of Behrend function comes from learning Denef-Loeser's motivic zeta function and analytic Milnor fiber, 
we used Chatgpt-5.6-Luna to help  calculating the families of examples of ideals. 
This work is partially supported by NSF DMS-2401484, and the College Black-Babcock Professorship Funds from University of Kansas.

\section{The trace formula}\label{sec_trace_jet}

In this section we recall a result of Denef-Loeser using jet-space to calculate the Euler characteristic of the Milnor fiber. 

Let $f: U\to \C$ be a holomorphic function, and $P\in \Crit(f)\subset f^{-1}(0)$. For $\epsilon>0$ small
enough, one may consider the corresponding closed $\epsilon$-ball $B(P;\epsilon)$ around $P$.
For $\eta> 0$ we denote by $D_{\eta}$ the closed disk of radius $\eta$ around the origin in $\C$.
By Milnor’s local fibration Theorem, there exists $\epsilon_0> 0$ such that, for every
$0<\epsilon<\epsilon_0$, there exists $0 < \eta < \epsilon$ such that the morphism $f$ restricts to a fibration,
called the Milnor fibration,
$$
B(P,\epsilon)\cap f^{-1}(D_{\eta}\setminus \{0\})\to D_{\eta}\setminus \{0\}.
$$
Set-theoretically the Milnor fiber at $P$ is:
$$
F_P =f^{-1}(\eta)\cap B(P,\epsilon).
$$

Let $T_P$ be the local monodromy on $H^*(F_P,\Q)$ and put
\[
 L(T_P^m)=\sum_j(-1)^j\Tr(T_P^m\mid H^j(F_P,\Q)).
\]
For $m\geq1$ define the finite-dimensional algebraic space
\begin{equation}\label{eq:jets}
 \mathcal X_{m,P}(f)=
 \{\gamma:\operatorname{Spec}\C[t]/(t^{m+1})\to M:
 \gamma(0)=P,\quad f(\gamma(t))=t^m\bmod t^{m+1}\}.
\end{equation}
Denef--Loeser's theorem \cite{DL}, \cite[Theorem 1.1.1]{HL_arxiv}, \cite[Theorem 2.3]{Loeser}  gives
\begin{equation}\label{eq:DL}
 L(T_P^m)=\chi_c(\mathcal X_{m,P}(f)).
\end{equation}
In particular, if $\operatorname{mult}_P(f)=\mu$, the jet space is empty
for $0<m<\mu$, so
\begin{equation}\label{eq:AC}
 L(T_P^m)=0\quad(0<m<\mu).
\end{equation}
For $m=1$ this is A'Campo's theorem $L(T)=0$ at a critical point, see \cite{AC}. 
The multiplicity refinement is due to Deligne.

\begin{theorem}\label{thm:jets}
Suppose $T_P^m$ is unipotent on every $H^j(F_P,\Q)$. Then
\begin{equation}\label{eq:mainjet}
 \chi(F_P)=\chi_c(\mathcal X_{m,P}(f)).
\end{equation}
Consequently $\chi(F)=0$ if and only if this jet space has Euler
characteristic zero. In particular, $m<\operatorname{mult}_P(f)$ is a
sufficient condition.
\end{theorem}
\begin{proof}
A unipotent linear operator has trace equal to the dimension of its
vector space. Therefore $L(T_P^m)=\chi(F_P)$, and~\eqref{eq:DL} proves
the assertion. The last assertion follows from emptiness of the jet
space.
\end{proof}

Local monodromy is quasi-unipotent, so such an $m$ exists; it may be
taken to be a common multiple of the orders of all monodromy
eigenvalues. The theorem is useful when a small $m$ is independently
known. It does not provide a uniform small $m$ at arbitrary Hilbert
points, nor a claim of finite determinacy of the entire singularity.
The monodromy hypothesis is essential to truncating at order $m$.

\begin{corollary}\label{cor:unipotent}
At a critical point, unipotent local monodromy on all cohomology groups
implies $\chi(F_P)=0$. Hence on the trace chart of $X_n$ it implies
$\nu_{X_n}(P)=(-1)^n$.
\end{corollary}
\begin{proof}
Take $m=1$ and use $df(P)=0$, which makes $\mathcal X_{1,P}(f)$ empty.
Equivalently, $\chi(F_P)=L(T_P)=0$ by A'Campo's theorem.
\end{proof}

Unipotence is sufficient, not necessary. For instance,
$f=(xy)^m$ with $m>1$ has a Milnor fibre consisting of $m$ annuli.
Monodromy permutes these components and is not unipotent, but the
Euler characteristic is zero.

\section{The local invariants on $X_n$}\label{sec:local:invariants}
\subsection{A result for $X_n$}
Recall that $X_n=\Crit(f_n)$, where $f_n: M\to \C$ is given by $\mbox{tr}(a[b,c])$, and $M$ admits a lift $\widetilde{M}\subset E\oplus \C^n$ for 
$E=\operatorname{End}(\C^n)^3$.  Thus, $\widetilde M\subset E\oplus\C^n$ is the open set of cyclic commuting 
quadruples. Its map 
$$\widetilde M\to M$$
to $M$ is a smooth principal
$\operatorname{GL}_n$-bundle, of relative dimension $n^2$.

\begin{proposition}\label{prop:matrixlift}
The cohomology of the Milnor fiber $F_P$ and its monodromy at $P\in M$ agree with those
of the trace polynomial on $E$ at $(A_0,B_0,C_0)$. Its Hessian rank is
$r=N-\dim T_P X_n$.
\end{proposition}
\begin{proof}
Analytically near the chosen point $P$, $\widetilde M$ is a product with a smooth
$n^2$-dimensional germ, and its potential is the pullback of $f_n$.
Adding such smooth variables leaves the ordinary Milnor cohomology
and monodromy unchanged. On $\widetilde M$ the trace function is
independent of the vector $v$; deleting those $n$ free local variables
therefore leaves the same invariant on $E$.

Pairwise commutativity and cyclic invariance of trace make the
constant and linear Taylor terms vanish. Direct expansion gives
\eqref{eq:exacttrace}. The critical scheme upstairs is the pullback
of $X_n$, so its tangent dimension is $\dim T_PX_n+n^2$.
Its Hessian rank is consequently
\[
 (3n^2+n)-(\dim T_PX_n+n^2)=N-\dim T_PX_n.
\]
Deleting the free vector variables does not change the rank.
\end{proof}

\subsection{The formula for every positive integer \texorpdfstring{$m$}{m}}
For $m\ge2$, introduce $m-1$ matrix triples
$z_i=(a_i,b_i,c_i)\in E$ and, for $2\le\ell\le m$, set
\begin{align}
 P_{\ell,P}(z_1,\ldots,z_{m-1})
 ={}&\sum_{\substack{i,j\ge1\\i+j=\ell}}
 \Tr\bigl(A_0[b_i,c_j]+a_i[B_0,c_j]+a_i[b_j,C_0]\bigr)
 \nonumber\\
 &+\sum_{\substack{i,j,k\ge1\\i+j+k=\ell}}
 \Tr\bigl(a_i[b_j,c_k]\bigr).
 \label{eq:allmatrixcoeff}
\end{align}
Only indices at most $m-1$ occur. Define the affine algebraic locus
\begin{equation}\label{eq:Ymatrix}
 Y_{m,P}=\{P_{2,P}=\cdots=P_{m-1,P}=0,\ P_{m,P}=1\}
 \subset E^{m-1}.
\end{equation}
For $m=2$ the list of zero equations is empty; put $Y_{1,P}=\varnothing$.

\begin{theorem}\label{thm:allmatrix}
For every positive integer $m$,
\begin{equation}\label{eq:allL}
  L(T_P^m)=\chi_c(Y_{m,P}).
\end{equation}
If $(T_P^m-I)^q=0$ on every cohomology group, then
\begin{equation}\label{eq:allchi}
  \chi(F_P)=\chi_c(Y_{m,P}),\qquad
 \nu_{X_n}(P)=(-1)^n(1-\chi_c(Y_{m,P})).
\end{equation}
\end{theorem}
\begin{proof}
Write a based matrix jet as
\[
 A(t)=A_0+\sum_{i=1}^m a_i t^i,
 \quad B(t)=B_0+\sum_{i=1}^m b_i t^i,
 \quad C(t)=C_0+\sum_{i=1}^m c_i t^i.
\]
The coefficient of $t^\ell$ in $\Tr(A(t)[B(t),C(t)])$ is
$P_{\ell,P}$. Terms containing a single positive-index matrix vanish
by commutativity of the base triple. Therefore the coefficient of
$t^m$ does not involve $a_m,b_m,c_m$.
The order-$m$ contact jet space on $E$ is exactly
\[
 Y_{m,P}\times E.
\]
Its Euler characteristic is $\chi_c(Y_{m,P})$. If working instead
on the cyclic quadruple space, vector jets contribute the additional
free factor $\A^{nm}$. Based jets remain in the cyclic open set
because their center is cyclic and  they need not satisfy commuting
matrix equations. The contact condition is the scalar trace condition.

Apply the Denef--Loeser trace formula and
Proposition~\ref{prop:matrixlift}. Under the nilpotence condition,
$$\Tr(T_P^m|H^j)=\dim H^j$$ 
giving \eqref{eq:allchi}.
The case $m=1$ is the empty first-contact locus at a critical point.
\end{proof}

\section{The proof of Theorem \ref{thm:intromain}}\label{sec:proof_1.1}

Write $X_n=\Hilb^n(\C^3)$. Its trace-potential presentation is
\[
 X_n=\Crit(f_n)\subset M,\qquad
 f_n(A,B,C,v)=\Tr(A[B,C]),\qquad N=\dim M=2n^2+n.
\]
Here $M$ is the smooth quotient of cyclic quadruples in
$\operatorname{End}(\C^n)^3\oplus\C^n$ by $\operatorname{GL}_n$.
For $P\in X_n$, let $F_{P}$ denote the local Milnor fiber. Behrend's formula gives
\begin{equation}\label{eq:behrend}
 \nu_{X_n}(P)=(-1)^n(1-\chi(F_{P})).
\end{equation}
Thus the expected value $(-1)^n$ is equivalent to $\chi(F_{P})=0$.

\subsection{The square-monodromy criterion}

Choose local coordinates centred at $P$ and expand
\[
 f_n(z)=Q(z)+C(z)+\text{terms of order at least }4,
\]
where $Q$ and $C$ are homogeneous of degrees $2$ and $3$ respectively.
Let $r$ be the rank of $Q$, which is  the Hessian rank.

\begin{theorem}\label{thm:square}
For any critical germ $(M,P)$, $L(T_P^2)=1-(-1)^{r}$. If $T_P^2$ is unipotent on $H^\bullet(F_P,\Q)$, then
\begin{equation}\label{eq:square}
 \chi(F_P)=
 \begin{cases}
 0,&r\text{ even},\\
 2,&r\text{ odd}.
 \end{cases}
\end{equation}
For the critical scheme $X_n=\Crit(f_n)$ this gives
\begin{equation}\label{eq:parity}
 \nu_{X_n}(P)=(-1)^{\dim_{\C}T_PX_n}.
\end{equation}
\end{theorem}
\begin{proof}
A two-jet has the form $\gamma(t)=t v_1+t^2v_2$. Since $df(P)=0$,
\[
 f_n(\gamma(t))=t^2Q(v_1)\bmod t^3.
\]
Thus $\mathcal X_{2,P}(f_n)\simeq\{Q=1\}\times\A^N$, and
Theorem~\ref{thm:jets} gives $\chi(F_P)=\chi_c(\{Q=1\})$.
For $r=0$ this space is empty. For $r>0$, diagonalize $Q$; its level
set is an affine factor times the smooth affine quadric
$z_1^2+\cdots+z_{r}^2=1$. The latter retracts onto $\mathbb{S}^{r-1}$ and has
Euler characteristic $1+(-1)^{r-1}$. This proves~\eqref{eq:square}.

The linearization of $df=0$ at $P$ is the Hessian map, so
$\dim T_PX_n=N-r$. Therefore
\[
 \nu_{X_n}(P)=(-1)^N(1-\chi(F_P))
 =(-1)^{N+r}=(-1)^{N-r}.
\]
\end{proof}

\begin{corollary}\label{cor:hilbsquare}
If the local monodromy at $P\in X_n$ satisfies that $T_P^2$
is unipotent, then
\[
 \chi(F_{P})=0
 \quad\Longleftrightarrow\quad
 \dim T_P X_n\equiv n\pmod2.
\]
\end{corollary}

This implication is not asserted without the monodromy assumption.
In general, tangent-space dimension does not determine the Behrend
function. For example, $f=z^3$ and $f=z^4$ both give one-dimensional
Zariski tangent spaces at the origin of their critical schemes, but
the Behrend values are $2$ and $3$ respectively.

\subsection{The cubic-monodromy criterion}

Let $B$ be the symmetric bilinear form with $Q(v)=B(v,v)$, and put
\[
 K=\ker B,\qquad c=C|_{K}.
\]
For a critical scheme, $K=T_P\Crit(f)$. The cubic $c$ is well-defined
up to the induced linear change of coordinates on $K$: the change in
the cubic term from a quadratic coordinate substitution is of the form
$2B(v,h_2(v))$, which vanishes on $K$.

\begin{theorem}\label{thm:cube}
For any critical germ $(M,x)$, $L(T_P^3)=3\chi_c(\PP(K)\setminus V(C))$. If $T_P^3$ is unipotent on all cohomology groups, then
\begin{equation}\label{eq:cube}
 \chi(F_P)=\chi_c(\{v\in K:c(v)=1\})
 =3\chi_c\bigl(\PP(K)\setminus V(c)\bigr).
\end{equation}
In particular, $\chi(F_P)=0$ if and only if the indicated projective
cubic complement has Euler characteristic zero. If $c=0$, the
complement and the affine level set are empty, so vanishing follows.
\end{theorem}
\begin{proof}
For $\gamma(t)=t v_1+t^2v_2+t^3v_3$,
\[
 f_n(\gamma(t))=t^2Q(v_1)
 +t^3\bigl(2B(v_1,v_2)+C(v_1)\bigr)\bmod t^4.
\]
The jet equations are
\[
 Q(v_1)=0,\qquad2B(v_1,v_2)+C(v_1)=1,
\]
with $v_3$ arbitrary. Split the first equation into $v_1\in K$ and
$v_1\notin K$.

On $K$, the equations reduce to $c(v_1)=1$, and both $v_2$ and $v_3$
are free. Off $K$, the second equation is a nonzero linear equation
in $v_2$, so its fibres are affine spaces of dimension $N-1$.
Their contribution to Euler characteristic is
$\chi_c(\{Q=0\}\setminus K)$. Choose a linear complement $L$ of
$K$, on which $Q$ is nondegenerate. Then
\[
 \{Q=0\}\setminus K\simeq
 K\times\{\ell\in L\setminus\{0\}:Q(\ell)=0\}.
\]
The second factor is the complement of the zero section in the
tautological line bundle over a projective quadric. It has Euler
characteristic zero because $\chi_c(\C^*)=0$; empty cases give the
same conclusion. Only the contribution from $K$ survives. Apply
Theorem~\ref{thm:jets} with $m=3$.

Finally, projectivization $\{c=1\}\to\PP(K)\setminus V(c)$ is a
finite \emph{\'etale} cover of degree $3$. Euler characteristic
multiplies by its degree. If $c=0$ both spaces are empty.
\end{proof}

\subsection{Fourth powers: a corrected quartic on the cubic critical cone}

Write the Taylor expansion in coordinates on $V=T_PM$ as
\[
 f_n=Q+C+D_x+O(5),
\]
with homogeneous terms of degrees $2,3,4$. Let $Q(v)=B(v,v)$,
$K=\ker B$, and choose a complement $L$ so that $V=L\oplus K$.
Put $r=\dim L$, $C=C|_{K}$, and let
\[
 \mathsf B:L\longrightarrow L^*,\qquad \mathsf B(u)=B(u,-)
\]
be the isomorphism defined by the nondegenerate quadratic form.
For $v\in K$ set
\[
 a(v)=d(C)_v|_{L} \text{~in~} L^*,\qquad
 h(v)=D_x(v)-\frac14 a(v)\bigl(\mathsf B^{-1}a(v)\bigr).
\]
Here $a$ is quadratic in $v$, so $h$ is homogeneous quartic. If $r=0$,
the correction is zero. Define the reduced cubic critical cone and
its projective quartic complement by
\[
 \Sigma=\{v\in K:dc_v=0\},\qquad
 B_4=\PP(\Sigma)\setminus V(h),\qquad b_4=\chi_c(B_4).
\]
We mean $\PP(\Sigma)=(\Sigma\setminus\{0\})/\C^*$.
In particular $\PP(\{0\})$ is empty. Only underlying reduced spaces
enter these Euler characteristics.

\begin{theorem}\label{thm:fourth}
For every critical germ, independently of a monodromy hypothesis,
\begin{equation}\label{eq:L4}
 L(T_P^4)=1+(-1)^{r}(4b_4-1).
\end{equation}
If $T_P^4$ is unipotent on every cohomology group, then
\begin{equation}\label{eq:fourth}
 \chi(F_P)=
 \begin{cases}
 4b_4,&r\text{ even},\\
 2-4b_4,&r\text{ odd}.
 \end{cases}
\end{equation}
Consequently $\chi(F_P)=0$ if and only if $r$ is even and $b_4=0$.
Moreover
\begin{equation}\label{eq:nu4}
 \nu_{X_n}(P)=(-1)^{N+r}(1-4b_4).
\end{equation}
\end{theorem}

\begin{proof}
A four-jet is $\gamma=t v+t^2w+t^3z+t^4e$. The contact equations are
\begin{align*}
 Q(v)&=0,\\
 2B(v,w)+C(v)&=0,\\
 2B(v,z)+Q(w)+d(C)_v(w)+D_x(v)&=1.
\end{align*}
The vector $e$ is arbitrary. If $v\notin K$, the second and third
conditions give affine hyperplanes successively in $w$ and $z$.
Their contribution is $\chi_c(\{Q=0\}\setminus K)=0$, by the
punctured-cone argument in the cube proof.

If $v\in K$, the second condition is $c(v)=0$.
Write $w=u+w_K$ with $u\in L$ and $w_K\in K$.
If $dc_v\ne0$, the last equation is a nonzero linear equation in
$w_K$, after choosing $u$. This contribution is
\[
 \chi_c\{v\in K:c(v)=0,\ dc_v\ne0\}=0,
\]
since this is a scaling-invariant constructible subset avoiding the
vertex, with a $\C^*$-bundle over its projectivization.

It remains to take $v\in\Sigma$; Euler's homogeneous identity then
also gives $c(v)=0$. Completing the square in $u$ gives
\[
 Q(u)+a(v)(u)+D_x(v)
 =Q\left(u+\frac12 \mathsf B^{-1}a(v)\right)+h(v).
\]
All other jet variables are free. Thus the remaining Euler
characteristic is that of
\[
 \{(v,u)\in\Sigma\times L:Q(u)=1-h(v)\}.
\]
An affine quadratic cone has Euler characteristic $1$, while a
nonzero level of $Q$ has Euler characteristic $1-(-1)^r$.
These formulas also hold for $r=0$, when the nonzero level is empty.
The cone $\Sigma$ has Euler characteristic $1$: its vertex contributes
one and its complement is a $\C^*$-bundle.
Set $A_4=\{v\in\Sigma:h(v)=1\}$. Euler integration along the last
projection therefore gives
\[
 \chi_c(A_4)+(1-(-1)^r)(1-\chi_c(A_4))
 =1+(-1)^r(\chi_c(A_4)-1).
\]
Projectivization $A_4\to B_4$ is finite \'etale of degree four, so
$\chi_c(A_4)=4b_4$. Apply the Denef--Loeser trace formula.
If $T_P^4$ is unipotent, $L(T_P^4)=\chi(F_P)$. Since $b_4$ is an integer,
$2-4b_4$ can never be zero. Behrend's formula proves \eqref{eq:nu4}.
\end{proof}

\begin{corollary}\label{cor:fourthhilb}
Suppose $T_{x}^4$ is unipotent and put
$r=N-\dim T_PX_n$. Then
\[
 \nu_{X_n}(P)=(-1)^{\dim T_PX_n}(1-4b_4).
\]
The expected value $(-1)^n$ occurs precisely when $r$ is even and
$b_4=0$. In particular, odd Hessian rank forces nonconstancy relative
to the expected value at that point without knowing $b_4$.
\end{corollary}

\section{Torus actions and  the monodromy power}\label{sec:torus_action}

In this section we give a criterion of the unipotency of the local monodromy using torus action. 

\subsection{Torus action}
\begin{proposition}\label{prop:torus}
Suppose a $\C^*$-action on the smooth germ $(M,P)$ that fixes $P$ and let $f_n: M\to \C$ be the potential function.  
If we have
\[
 f_n(t\cdot u)=t^q f_n(u),\qquad q\in\mathbb{Z}\setminus\{0\}.
\]
Then the local monodromy has a representative of order dividing $|q|$. 
In particular $T_P^{|q|}=1$ on the cohomology.
\end{proposition}
\begin{proof}
Restrict to the compact subgroup $\mathbb{S}^1$. Holomorphic linearization of a
compact group at a fixed point gives local coordinates with a unitary
$\mathbb{S}^1$-action, and hence an invariant sufficiently small ball. Transport
around the target circle is given by $u\mapsto e^{2\pi i s/q}\cdot u$
for $0\leq s\le q-1$. This preserves the ball and maps each Milnor fibre
to the fibre over $e^{2\pi i s/q}f_n(u)$. The return map is the action of
$e^{2\pi i/q}$, whose $|q|$th power is the identity.
\end{proof}

Weights need not be positive. Fixing the centre $P$ is essential:
an action that moves $P$ gives no such trivialization of its local
Milnor fibration. In particular, a global weight-one action on $f_n$
does not prove local unipotence at every Hilbert point.

The coordinate torus acts on $M$ by scaling $(A,B,C)$, and the
potential has character $t_1t_2t_3$. Suppose an ideal $I$ is invariant
under
\[
 (x,y,z)\longmapsto(t^a x,t^b y,t^c z).
\]
Then $P=[I]$ is fixed and the potential has weight $q=a+b+c$.
We obtain the following concrete tests:
\begin{center}
\begin{tabular}{p{0.23\textwidth}p{0.66\textwidth}}
\toprule
Hypothesis  & Consequence at $P$\\
\midrule
$|a+b+c|=1$ & $\chi(F_{P})=0$ and $\nu(x)=(-1)^n$.\\[3pt]
$|a+b+c|=2$ & $\chi(F_{P})=0$ precisely when
$\dim T_PX_n\equiv n\pmod2$.\\[3pt]
$|a+b+c|=4$ & Vanishing is equivalent to even Hessian rank and $b_4=0$ from Theorem~\ref{thm:fourth}.\\[3pt]
$|a+b+c|=3$ & Vanishing is equivalent to
$\chi_c(\PP(T_PX_n)\setminus V(c))=0$, with $c$ as above.\\
\bottomrule
\end{tabular}
\end{center}

Every monomial ideal is invariant under the grading $(1,0,0)$, so the
first row recovers the expected Behrend value for monomial ideals.
More generally it applies to any ideal homogeneous for an integral
grading whose three weights sum to $1$ or $-1$.
Every standard homogeneous ideal is fixed by weights $(1,1,1)$.
Consequently the weight-three row applies to every standard homogeneous
ideal, i.e.,  its Behrend value is controlled by a cubic on the tangent space.

\subsection{Jelisiejew-Kool-Schmiermann length-$24$ example}

The ideal
\[
 I=\bigl((x^2)+(y,z)^2\bigr)^2+(y^3-x^3z)
\]
is homogeneous for weights $(1,1,0)$. Its trace potential has weight
$2$, so Proposition~\ref{prop:torus} gives $T_P^2=1$ at $[I]$.
Using the tangent dimension $99$ established in
\cite[proof of Theorem 6]{JKS},
\[
 N_{24}=1176,\qquad r=1176-99=1077.
\]
Theorem~\ref{thm:square} therefore yields $\chi(F_{[I]})=2$ and
$\nu([I])=-1$. 

\begin{remark}
The above calculation  also determines the  monodromy zeta function in the Appendix. Write
$e_+$ and $e_-$ for the alternating dimensions of the $+1$ and $-1$
eigenspaces. A'Campo's formula gives $e_+-e_-=0$, while the Euler characteristic
gives $e_++e_-=2$. Hence
\begin{equation}\label{eq:examplezeta}
 e_+=e_-=1,\qquad \zeta_{f,[I]}(t)=\frac1{1-t^2}
\end{equation}
where the last formula is the monodromy zeta function in the Appendix.
\end{remark}

\section{Ideals with unipotent second, third and fourth powers}
\label{sec:examples}

In this section we supply explicit ideals, and prove Theorem \ref{thm:family}.  All identities $T_P^q=1$ below refer to ordinary
Milnor cohomology of the standard trace potential, including $H^0$.

\subsection{Trivial monodromy examples}

For a weighted-homogeneous ideal under weights $(a,b,c)$, the point in
$X_n$ is fixed by the corresponding one-parameter subgroup. The
potential has weight $a+b+c$. Proposition~\ref{prop:torus} gives
\[
 a+b+c=q\ne0\quad\Longrightarrow\quad T_P^{|q|}=1.
\]
Each monomial generator is homogeneous for every diagonal grading.
For a binomial generator it suffices that its two monomials have the
same weight. Different generators need not have the same weight.
This last observation is important for the examples below.

Every finite-colength monomial ideal admits weights $(1,0,0)$, so
$T_P=1$. For instance
\[
 I_{a,b,c}=(x^a,y^b,z^c),\qquad
 \operatorname{length}(\C[x,y,z]/I_{a,b,c})=abc
\]
has $T_P=1$, $\chi(F_P)=0$, and $\nu_{X_n}=(-1)^{abc}$.

\subsection{A common finite-flat construction}
Use separate variables $u,v,w$ and the ideal
\begin{equation}\label{eq:seed12}
 J=\bigl((u)+(v,w)^2\bigr)^2+(v^3-uw)\subset R=\C[u,v,w].
\end{equation}
From  \cite[Equation (1.1)]{GGGL}, we have
\[
 \dim R/J=12,\qquad \dim\operatorname{Hom}_R(J,R/J)=45.
\]
This is the example in  \cite[Equation (1.1)]{GGGL} that disproves the parity conjecture of tangent spaces which is given by 
$$
\dim_{\C}T_{X_n}|_{P}\equiv n\mod (2).
$$

For positive integers $p,q,r$, define in $S=\C[x,y,z]$
\begin{equation}\label{eq:flatfamily}
 I_{p,q,r}
 =\bigl((x^p)+(y^q,z^r)^2\bigr)^2+(y^{3q}-x^p z^r).
\end{equation}
These ideals are all supported at the origin.

\begin{proposition}\label{prop:flatfamily}
The ideal $I_{p,q,r}$ has colength and tangent dimension
\[
 n=12pqr,\qquad \dim T_{[I_{p,q,r}]}X_n=45pqr.
\]
If integers $(\alpha,\beta,\gamma)$ satisfy
\[
 p\alpha+r\gamma=3q\beta,
\]
and $d=\alpha+\beta+\gamma\ne0$, then $T_P^{|d|}=1$ at this ideal.
\end{proposition}
\begin{proof}
The map $R\to S$ given by
$u\mapsto x^p$, $v\mapsto y^q$, $w\mapsto z^r$ makes $S$ finite free
of rank $pqr$ over $R$, with basis $x^iy^jz^k$ for
$0\le i<p$, $0\le j<q$, $0\le k<r$.
The extended ideal is $JS=I_{p,q,r}$. Flatness and finite presentation
give
\[
 \operatorname{Hom}_S(JS,S/JS)
 \simeq \operatorname{Hom}_R(J,R/J)\otimes_R S.
\]
For completeness this base-change identity follows by taking a finite
presentation of $J$, expressing Hom as a kernel between finite powers
of $R/J$, and using flatness of $S$ to preserve that kernel.
Both dimension formulas now follow from the free basis.
The quotient is supported at the origin because it contains powers
of $x,z$, and its binomial relation then forces a power of $y$ to
vanish. The grading equation makes the binomial homogeneous and  the
other generators are monomials. Apply the grading certificate gives the result.
\end{proof}

Let $g=\gcd(p,3q,r)$. The image of the trace character on the
cocharacter lattice of the connected diagonal stabilizer is
$D$, where
\begin{equation}\label{eq:stabilizerD}
 D=\frac{\gcd(p+3q,r+3q)}{g}.
\end{equation}
Indeed that lattice is the kernel of the primitive row
$(p,-3q,r)/g$, and the trace character is the row $(1,1,1)$.
The gcd of the two-by-two minors of these two rows is $D$.
Since the first row is primitive, this gcd is exactly the index of
the image of the second row on the kernel. Equivalently this follows
by the Smith normal form. Thus a grading with trace weight $D$
exists and gives $T^D=1$.
This computes a bound certified by diagonal gradings, not the exact
order of monodromy. In this family the only condition on the
connected diagonal stabilizer is the displayed binomial relation:
its two monomials survive independently modulo the monomial part.

\subsection{Second powers: the proof of (1) and (i) in Theorem \ref{thm:family}}
The ideals $P\in X_n$ is given as follows. 
Set
\[
P= I_m=I_{3,1,m}
 =\bigl((x^3)+(y,z^m)^2\bigr)^2+(y^3-x^3z^m),\qquad m\ge1.
\]
Weights $(1,1,0)$ give trace weight two. Let $\tau$ represent the tangent dimension. Therefore
\[
 T_P^2=1,\quad n=36m,\quad \tau=135m,
 \quad \nu_{X_n}(I_m)=(-1)^m,\quad \chi(F_{I_m})=1-(-1)^m.
\]
The last two identities use the square criterion.
If $m$ is odd, $\chi(F_P)=2$; A'Campo's $L(T_P)=0$ excludes $T_P=1$.
Thus the cohomological monodromy has \emph{exact order two} for odd
$m$. For even $m$, weights
\[
 \left(-\frac{m+2}{2},\frac{m-2}{2},3\right)
\]
have sum one and preserve the ideal, so in fact $T_P=1$.
Here is the table of $m=1, 2, 3, 4, 5$ for $P=I_m$ in $X_n$:
\begin{center}
\begin{tabular}{rrrrl}
\toprule
$m$ & length & tangent dimension & $\chi(F_P)$ & monodromy order\\
\midrule
1&36&135&2&2\\
2&72&270&0&1\\
3&108&405&2&2\\
4&144&540&0&1\\
5&180&675&2&2\\
\bottomrule
\end{tabular}
\end{center}

The  ideal in \cite{JKS}
\[
 I_{\mathrm{JKS}}=\bigl((x^2)+(y,z)^2\bigr)^2+(y^3-x^3z)
\]
also has exact order two: its length is $24$, its tangent dimension
is $99$, weights $(1,1,0)$ give $T_P^2=1$, and $\chi(F_P)=2$.
It must not be confused with $I_{2,1,1}$, whose binomial is instead
$y^3-x^2z$ and whose tangent dimension is $90$.

\subsection{Third powers: the proof of (2) in Theorem \ref{thm:family}}
For $d\ge2$, put $\mathfrak m=(x,y,z)$ and
\begin{equation}\label{eq:fermatideals}
 B_d=\mathfrak m^{d+1}+(x^d+y^d+z^d).
\end{equation}
Its length is
\[
 n_d=\binom{d+3}{3}-1.
\]
To see this, all monomials of total degree at most $d$ survive modulo
$\mathfrak m^{d+1}$, except for one linear relation in degree $d$.
Multiples of that relation in positive degree are already zero.
Weights $(1,1,1)$ fix $B_d$ and give $T_P^3=1$.
Thus the first three examples have lengths $9,19,34$, respectively.
The cube formula gives
\[
 \chi(F_{B_d})=3b_{3,B_d},\qquad
 \nu_{X_n}(B_d)=(-1)^{n_d}(1-3b_{3,B_d}).
\]
No numerical value of $b_{3,B_d}$ is asserted here.
For a diagonal one-parameter subgroup to preserve the degree-$d$
relation, its weights must satisfy $d\alpha=d\beta=d\gamma$.
Thus its nonzero trace weights are multiples of three. This excludes
a weight-one certificate in the diagonal torus, but does \emph{not}
prove $T_P\ne 1$.

More generally, for any vector subspace
$W\subset\C[x,y,z]_d$, the ideal $W+\mathfrak m^{d+1}$ has
length $\binom{d+3}{3}-\dim W$ and satisfies $T_P^3=1$.

The finite-flat construction gives another family with known tangent
dimensions:
\[
 C_s=I_{3s,2s,3s}
 =\bigl((x^{3s})+(y^{2s},z^{3s})^2\bigr)^2
 +(y^{6s}-x^{3s}z^{3s}),\qquad s\ge1.
\]
It is standard homogeneous, has length $216s^3$ and tangent dimension
$810s^3$, and satisfies $T_P^3=1$. Formula \eqref{eq:stabilizerD} gives
$D=3$. Again, the actual order may be one or three; that question is
not settled merely by the grading or tangent dimension.

\subsection{Fourth powers: the proof of (3) in Theorem \ref{thm:family}}
Set
\[
 D_s=I_{s,s,s}
 =\bigl((x^s)+(y^s,z^s)^2\bigr)^2+(y^{3s}-x^s z^s),\qquad s\ge1.
\]
Weights $(2,1,1)$ give $T_P^4=1$, and
\[
 n=12s^3,\qquad \tau=45s^3.
\]
The connected diagonal stabilizer has smallest positive trace weight
four, by \eqref{eq:stabilizerD}. This still does not establish exact
monodromy order four.

The fourth-power criterion gives, with its effective-quartic integer
$b_{4,D_s}$,
\[
 \chi(F_{D_s})=
 \begin{cases}4b_{4,D_s},&s\text{ even},\\
 2-4b_{4,D_s},&s\text{ odd}.
 \end{cases}
\]
In particular
\begin{equation}\label{eq:12congruence}
 s\text{ odd}\quad\Longrightarrow\quad
 \chi(F_{D_s})\equiv2\pmod4,\qquad \nu_{X_n}(D_s)\equiv-1\pmod4.
\end{equation}
The first lengths are $12,96,324,768$.
For odd $s$, $T_P$ cannot be unipotent, and its order is either two or
four. Nor can $T_P^3$ be unipotent: combined with $T_P^4=1$, that would
force $T_P=1$. No primitive fourth-root eigenspace is claimed to have
been computed here.

\begin{proposition}\label{prop:12deduction}
At the explicit seed $D_1=J$ (with variables renamed),
\[
 \nu_{X_{12}}(J)\equiv-1\pmod4.
\]
In particular its value is not $1$. The grading and trace identities
therefore imply nonconstancy on $X_{12}$, and adjoining disjoint
reduced points propagates this conclusion to every $X_n$, $n\ge12$.
 It does not compute the exact
integer $\nu_{X_{12}}(J)$.
\end{proposition}
\begin{proof}
We give an argument that does not require computing the effective
quartic. The ideal is fixed by weights $(2,1,1)$, so $T_P^4=1$.
Its Hessian rank is
\[
 (2\cdot12^2+12)-45=255.
\]
The two-jet trace identity gives $L(T_P^2)=2$.
For a rational representation with $T_P^4=1$, eigenvalues are
$1,-1,i,-i$, and the multiplicities of $i$ and $-i$ agree in each
cohomological degree. Let $e_+,e_-,e_i=e_{-i}$ be the alternating
multiplicities. Then
\[
 \chi(F_P)=e_++e_-+2e_i,\qquad
 L(T_P^2)=e_++e_--2e_i.
\]
It follows that $\chi(F_P)-L(T_P^2)=4e_i$, and hence $\chi(F_P)\equiv2$
modulo four. Since $n=12$ is even, $\nu_{X_n}=1-\chi(F_P)\equiv-1$ modulo
four. Reduced length-twelve subschemes have value $1$.
The product and \'etale properties of Behrend functions give the
last assertion after adding disjoint reduced points.
\end{proof}

\subsection{The proof of (ii), (iii) in Theorem \ref{thm:family}}

We evaluate the two Behrend function values $\nu_{X_{9}}(B_2)=2$ and $\nu_{X_{12}}(D_1)=3$.  From Theorem \ref{thm:intromain}, it is sufficient to evaluate
\begin{align*}
 b_{3,B_2}=1,&\qquad \chi(F_{B_2})=3;\\
 b_{4,D_1}=1,&\qquad \chi(F_{D_1})=-2.
\end{align*}
We divide the calculation into several steps:

\subsubsection{Localizing the polynomial calculations}
After a linear change of  coordinates, $B_2$ becomes
\begin{equation}\label{eq:B2prime}
 B'_2=\mathfrak{m}^3+(y^2-xz).
\end{equation}
All nondegenerate quadratic forms in three variables are equivalent
over $\C$, so $\nu_{X_{9}}(B'_2)=\nu_{X_{9}}(B_2)$. Both $B'_2$ and $D_1$ are fixed by a torus $T_0$ action
\[
 T_0:\quad (x,y,z)\longmapsto(tx,y,t^{-1}z).
\]
The trace potential is $T_0$-invariant.

For an invariant polynomial on a representation, Euler localization
on a projective complement leaves only the weight-zero subspace, and 
on a nonzero weight-$j$ subspace, an invariant homogeneous polynomial
of positive degree vanishes identically. We apply these to the numbers
$b_3, b_4$ in Theorem \ref{thm:intromain}, we get
\begin{align}
 b_3&=\chi_c\bigl(\PP(K^{T_0})\setminus V(c|_{K^{T_0}})\bigr),\label{eq:loc3}\\
 b_4&=\chi_c\bigl(\PP(\Sigma\cap K^{T_0})\setminus V(h|_{K^{T_0}})\bigr).
 \label{eq:loc4}
\end{align}
Moreover $\Sigma\cap K^{T_0}$ is the critical locus of
$c|_{K^{T_0}}$, and  all derivatives in moving directions vanish at
fixed vectors by invariance. One may choose the quadratic complement
$W$ invariant, and the effective quartic on $K^{T_0}$ is computed
entirely in the fixed smooth chart. Equivalently, Euler localization
on an invariant Milnor fibre gives
$\chi(F_f)=\chi(F_{f|_{U_n^{T_0}}})$.

\subsubsection{The matrix chart and its exact certificates}\label{sec:chart}

We specify the chart so the two polynomial calculations are
reproducible without choosing implicit quotient coordinates.
Use the following ordered monomial bases, indexed from zero:
\begin{align*}
 \mathcal B_B&=(1,z,y,x,z^2,yz,xz,xy,x^2),\\
 \mathcal B_D&=(1,z,y,x,z^2,yz,y^2,xz,xy,z^3,yz^2,y^2z).
\end{align*}
For each nonconstant basis monomial, select the first variable in
the order $x,y,z$ whose removal leaves a basis monomial. Freeze the
corresponding multiplication-matrix column to the appropriate unit
vector. Set the cyclic vector to $e_0$. These $n-1$ columns define a
tree of words; the condition that those words form a basis gives an
affine chart on $M$. The other matrix entries are independent.

Let $w_i$ be the $x$-exponent minus the $z$-exponent of the $i$th
basis monomial. In the $T_0$-fixed chart retain only entries
\[
 A_{ij}:w_i=w_j+1,\qquad B_{ij}:w_i=w_j,\qquad
 C_{ij}:w_i=w_j-1.
\]
At the center they are the multiplication matrices of the indicated
quotient algebra. Translate the unfrozen entries by these constants.
The potential has exactly a quadratic and a cubic term, since all
three matrices depend affinely on the coordinates.

Exact Hessian computations give

\begin{center}
\begin{tabular}{@{}lrrrr@{}}
\toprule
 & Full chart dimension & Full Hessian rank
 & Fixed chart dimension & Fixed Hessian rank\\
\midrule
$B'_2$ &171&128&35&28\\
$D_1$  &300&255&60&53\\
\bottomrule
\end{tabular}
\end{center}

In both fixed charts the Hessian kernel is seven-dimensional.
Here is an explicit parametrization by $u_0,\cdots,u_6$.
The entries below are first-order increments, and  every entry not listed
is zero. Matrix indices are zero-based and denote (row,column).

\begin{center}
\begin{tabular}{@{}lll@{}}
\toprule
 & $B'_2$ & $D_1$\\
\midrule
$\delta A$ &
$A_{5,4}=u_4,\ A_{6,5}=u_5$ &
$A_{11,4}=u_3,\ A_{7,5}=u_4$\\
&$A_{7,6}=u_6,\ A_{8,7}=u_2$ & $A_{8,6}=u_1$\\[2pt]
$\delta B$ & $B_{2,2}=u_1-u_6,\ B_{6,2}=u_0$
&$B_{6,6}=u_2-u_3,\ B_{7,6}=u_0$\\
&$B_{4,4}=u_3,\ B_{5,5}=u_1+u_4-u_6$
&$B_{7,7}=u_4,\ B_{8,8}=u_1$\\
&$B_{6,6}=u_5,\ B_{7,7}=u_1,\ B_{8,8}=u_2$
&$B_{9,9}=u_5,\ B_{10,10}=u_6,\ B_{11,11}=u_2$\\[2pt]
$\delta C$ & $C_{4,5}=u_3,\ C_{5,6}=u_4$
&$C_{11,7}=u_3,\ C_{7,8}=u_4$\\
&$C_{6,7}=u_5,\ C_{7,8}=u_6$
&$C_{9,10}=u_5,\ C_{10,11}=u_6$\\
\bottomrule
\end{tabular}
\end{center}

\subsubsection{The calculation for $B_2$}\label{sec:B2}

Restricting the cubic $\Tr(\delta A[\delta B,\delta C])$ to the
kernel in Section~\ref{sec:chart}, then setting
\[
 a=u_1-2u_5,\quad b=u_2-u_5,\quad c=u_3-u_5,
 \quad d=u_4,\quad e=u_6,
\]
gives the cubic in five independent variables
\begin{equation}\label{eq:cubic}
 P(a,b,c,d,e)=abe-acd-b^2e+c^2d-cd^2+cde.
\end{equation}
The other two kernel coordinates are absent. Projection away from
those two coordinates is an affine-plane bundle on the complement,
so \eqref{eq:loc3} reduces to
\[
 b_{3,B_2}=\chi_c(\PP^4\setminus V(P)).
\]

\begin{lemma}\label{lem:cubic-euler}
$\chi_c(\PP^4\setminus V(P))=1$.
\end{lemma}
\begin{proof}
Project from $[1:0:0:0:0]\in V(P)$ to $\PP^3_{[b:c:d:e]}$ so that $\PP^4\setminus V(P)$ is decomposed into fibers over $\PP^3_{[b:c:d:e]}$. Since the point at infinity is always in \(V(P)\), the fibers are subsets of the affine line \(\mathbb{C}_{a}\).
Write
\[
 P=aA+B,\qquad A=be-cd,\qquad B=-b^2e+c^2d-cd^2+cde.
\]
The fibre over a point with $A\ne0$ is an affine line minus one
point, of Euler characteristic zero. Over $A=0,B\ne0$ it is an
affine line, of characteristic one; over $A=B=0$ it is empty.
On the quadric $A=0$,
\[
 B=cd(-b+c-d+e).
\]
Thus we need $c,d\ne0$, which also forces $b,e\ne0$.
Normalize $d=1$, set $b=u$, $c=v$, and hence $e=v/u$.
The last nonvanishing condition becomes
\[
 -u+v-1+v/u=\frac{(u+1)(v-u)}{u}\ne0.
\]
Using $v/u$ as the second coordinate, this locus is
\[
 (\C^*\setminus\{-1\})\times(\C^*\setminus\{1\}).
\]
Its compactly supported Euler characteristic is $(-1)^2=1$.
\end{proof}

Since $T_{B_2}^3=1$, Theorem \ref{thm:intromain} gives $\chi(F_{B_2})=3$.
The sign in \eqref{eq:behrend} is $(-1)^9=-1$, so
\[
\nu_{X_9}(B_2)=-(1-3)=2.
\]

\subsubsection{The quartic calculation for $D_1$}\label{sec:D1}

The cubic restricted to the fixed Hessian kernel is identically
zero. This also follows from weights, and  the kernel coordinates have
weights $0,1,1,1,1,1,1$ under $(x,y,z)\mapsto(t^2x,ty,tz)$,
whereas the potential has weight $4$.
There is no ordinary degree-three monomial of weight $4$ in those
coordinates. Consequently the fixed critical cone is the whole
seven-dimensional kernel.

The original matrix chart has no quartic term. Its effective
quartic is nevertheless nonzero:
\[
 h=-\tfrac12\ell^t B_W^{-1}\ell.
\]
With the linear coordinates
\[
 a=u_1-u_6,\quad b=u_2-\tfrac32u_6,\quad c=u_3,
 \quad d=u_4-\tfrac12u_6,\quad e=u_5-\tfrac12u_6,
\]
this quartic is
\begin{equation}\label{eq:quartic}
 Q(a,b,c,d,e)=\frac14
 \left(\bigl((a+c)(b-a)+(a-c)d-d^2+e^2\bigr)^2
       +4c(b-d)(d^2-e^2)\right).
\end{equation}
The two remaining kernel coordinates are again absent. The exact
Hessian-inverse calculation yielding \eqref{eq:quartic} is checked
using Python code, and  omitting this correction would give an
incorrect answer. Equation~\eqref{eq:loc4} gives
\[
 b_{4,D_1}=\chi_c(\PP^4\setminus V(Q)).
\]

\begin{lemma}\label{lem:quartic-euler}
$\chi_c(\PP^4\setminus V(Q))=1$.
\end{lemma}
\begin{proof}
Project from $[0:1:0:0:0]\in V(Q)$ to $\PP^3_{[a:c:d:e]}$.
As a polynomial in $b$, write $Q=A b^2+B b+C$. Direct expansion
of \eqref{eq:quartic} gives
\begin{align*}
 A&=(a+c)^2/4,\\
 \Delta:=B^2-4AC
 &=-ac(d-e)(d+e)(a+c-d-e)(a+c-d+e).
\end{align*}
Where $A\ne0$, the fibre of the complement has characteristic
$-1$ if $\Delta\ne0$ and $0$ if $\Delta=0$.
Normalize $a+c=1$, and put $t=a$, $u=d+e$, $v=d-e$.
The locus $A\ne0,\Delta\ne0$ is
\[
 (\C\setminus\{0,1\})^3,
\]
whose characteristic is $-1$. Thus, base times fiber  contribution is therefore $1$.

On $a+c=0$, so $c=-a$, the other coefficients become
\[
 B=-a(d-e)(d+e),\qquad
 C=a^2d^2+\tfrac14(d^2-e^2)^2.
\]
The locus $B\ne0$ has fibres $\A^1$ minus one point, and contributes
zero. The locus $B=0,C\ne0$ is a disjoint union of three copies
of $\C^*$: in coordinates $[a:u:v]$, where $u=d+e$, $v=d-e$,
these are
\[
 \{a=0,uv\ne0\},\qquad
 \{u=0,av\ne0\},\qquad
 \{v=0,au\ne0\}.
\]
Every intersection has $C=0$ and is excluded. These three strata
also contribute zero. Where $B=C=0$ the fibre is empty.
The total is $1$.
\end{proof}

The full Hessian has odd rank $255$ (the fixed Hessian has odd rank
$53$). Since $T_{D_1}^4=1$, therefore
\[
 \chi(F_{D_1})=2-4b_{4,D_1}=2-4=-2.
\]
The length is even, so
\[
 \nu_{X_{12}}(D_1)=1-(-2)=3.
\]

\appendix

\section{The zeta function and weak monodromy conjecture}

In this appendix we collect some known facts of local motivic, local monodromy zeta function of singularity germs $(M,f, P)$ and related them to the weak monodromy conjecture in \cite{BMT}.
We denote the potential $f_n: M\to \C$ for the Hilbert scheme $X_n=\Hilb^n(\C^3)$ by $f$.

\subsection{Monodromy zeta function}
Use the local monodromy $T_P$ of the Milnor fiber $F_P$, we define
\begin{equation}\label{eq:zeta}
 \zeta_{f,P}(t)=\prod_j
 \det(1-tT_P\mid H^j(F_P,\Q))^{(-1)^{j+1}}.
\end{equation}
For a nonzero rational function $P(t)/Q(t)$ define
$\deg(P/Q)=\deg P-\deg Q$; cancellation does not change this number.
Since $T_P$ is invertible, each determinant has degree $\dim H^j(F_P)$.
Therefore
\begin{equation}\label{eq:zetacriterion}
 \quad\deg\zeta_{f,P}=-\chi(F_P),\qquad
 \chi(F_P)=0\ \Longleftrightarrow\ \deg\zeta_{f,P}=0.\quad
\end{equation}
This  uses degree, not just the existence or location of roots and poles.

Let
$e_\lambda=\sum_j(-1)^j\dim H^j(F_P,\C)_\lambda$, where generalized
eigenspaces are used. Then
\[
 \zeta_{f,P}(t)=\prod_\lambda(1-\lambda t)^{-e_\lambda},\qquad
 \chi(F_P)=\sum_\lambda e_\lambda,\qquad
 L(T_P^m)=\sum_\lambda\lambda^m e_\lambda.
\]
Jordan block sizes disappear from these expressions. At a critical point,
A'Campo imposes $\sum_\lambda\lambda e_\lambda=0$; this does not in
general force $\sum_\lambda e_\lambda=0$. It does so when $1$ is the
only eigenvalue, which explains Corollary~\ref{cor:unipotent}.

\subsection{A resolution criterion}

Take a log resolution $\pi:\widetilde M\to M$ of $f^{-1}(0)$.
Let $E_i$ be the components of the total transform, $N_i$ their
multiplicities, and
$E_i^\circ=E_i\setminus\bigcup_{j\ne i}E_j$. A'Campo's formula gives
\begin{equation}\label{eq:acampo}
 \zeta_{f,P}(t)=\prod_i(1-t^{N_i})^{-\chi_c(E_i^\circ\cap\pi^{-1}(P))}.
\end{equation}
Hence
\begin{equation}\label{eq:resolution}
 \chi(F_P)=\sum_iN_i\chi_c(E_i^\circ\cap\pi^{-1}(P)).
\end{equation}
The formula and its Lefschetz-number form are classical, see \cite{AC}, \cite{Loeser}. 
It supplies an exact geometric criterion. For example, if every
indicated stratum has a locally trivial $\C^*$-fibration, all summands
vanish. Cancellation between nonzero summands is also possible.

\subsection{The relevant mixed Hodge data}

The nearby-cycle mixed Hodge module determines $H^*(F_P)$ with
its monodromy, and therefore determines every $e_\lambda$ above.
Equivalently, form its alternating Hodge polynomial in terms $(u,v)$ on each generalized
eigenspace and specialize at $(u,v)=(1,1)$. The sum of these
specializations is $\chi(F_P)$.

For the perverse-normalized vanishing-cycle 
$\mathcal V=\varphi_f[-1]\Q_M^{X_n}[N]$, with the constant object written
in unshifted notation, the eigenwise stalk Euler numbers satisfy
\[
 \chi((i_x^*\mathcal V)_\lambda)
 =(-1)^{N-1}(e_\lambda-\delta_{\lambda,1}),
\]
where $i_x: x\hookrightarrow X_n$ is the inclusion.
The subtraction at eigenvalue $1$ is the constant term in the
nearby/vanishing-cycle triangle. Summing gives
$\chi(i_x^*\mathcal V)=(-1)^N(1-\chi(F_P))$.
Thus local mixed Hodge modules determine the answer, but weights or
Hodge types without alternating eigenwise multiplicities do not.

Unipotence here is a statement about the \emph{local} nearby-cycle
monodromy at the selected critical point. Qin--Zhang's comparison
concerns global relative cohomology and limiting structures at
infinity, see \cite[Theorem 1.1.1]{QZ}. 
Local unipotence cannot be inferred from global unipotence without an
additional theorem controlling the direct image and individual stalks.

\subsection{Two distinct zeta functions}

The \emph{local topological zeta function} is a different invariant:
\[
 Z_{\mathrm{top},f,P}(s)
 =\sum_{\varnothing\ne I}
 \chi_c(E_I^\circ\cap\pi^{-1}(P))
 \prod_{i\in I}\frac1{N_i s+\nu_i},
\]
where $\nu_i$ are log discrepancies. 
Here $\{E_i\}_{i\in \mathcal{I}}$ is the set of irreducible components of the exceptional divisors of the resolution $\pi: \widetilde M\to M$.  For $I\subset \mathcal I$, we set 
$$
E_I:=\bigcap_{i\in I}E_i
$$
and 
$$
E_I^{\circ}:=E_I\setminus \bigcup_{j\notin \mathcal I}E_j.
$$

Let $m_i$ be the multiplicity of the component $E_i$ in $E$. Let $m_I=\gcd(m_i)_{i\in \mathcal I}$.  Let $U$ be an affine Zariski open subset of $\widetilde M$ such taht on $U$, $f\circ\pi=uv^{m_I}$, with $u$ a unit in $U$ and $v$ a morphism from $U$ to $\C$.  The restriction of $E_I^{\circ}\cap U$, which we denote by $\widetilde E_I^{\circ}\cap U$, is defined by 
$$
\{(z,y)\in \C\times (E_I^{\circ}\cap U)| z^{m_I}=u^{-1}\}.
$$
The $E_I^{\circ}$ can be covered by the open subset $U$ of $\widetilde M$.  Glue  together all such constructions and get the Galois cover 
$$
\widetilde E_I^{\circ}\to E_I^{\circ}
$$
with Galois group $\mu_{m_I}$.  

The weak topological monodromy
conjecture predicts that a pole $s_0$ gives a monodromy eigenvalue
$e^{2\pi i s_0}$, with the precise formulation allowing an appropriate
point of the zero fibre (or a nearby point in local formulations).
It does not specify alternating eigenvalue multiplicities at a chosen
point, see \cite{DL_92}.  The hyperplane arrangement version is in \cite{BMT}, which is proved in \cite{DY}.

There are three obstacles to deriving $\chi(F_P)=0$ from this conjecture:
it is a pole-to-eigenvalue implication rather than a converse; it
does not control the alternating multiplicities $e_\lambda$; and the
eigenvalue need not be detected at the particular point under study.
Even strengthening the conclusion to that point would not resolve the
multiplicity issue.

\subsection{An explicit test}

On $\C^2$, compare $f=x^m$ and $g=(xy)^m$, with $m\geq2$, at the
origin. Both have critical points there and both have all $m$th roots
of unity among their monodromy eigenvalues. Directly,
\[
 \begin{array}{c|c|c|c}
 &\chi(F_P)&\zeta(t)&Z_{\mathrm{top},0}(s)\\\hline
 x^m&m&(1-t^m)^{-1}&(ms+1)^{-1}\\
 (xy)^m&0&1&(ms+1)^{-2}
 \end{array}
\]
For $x^m$, the fibre is $m$ disks, with cyclic permutation on $H^0$.
For $(xy)^m$, it is $m$ annuli, and the same permutation appears in
$H^0$ and $H^1$, so it cancels in the zeta function. The topological
zeta functions follow from the identity normal-crossing resolution.
Thus the same eigenvalue set and the same topological-zeta pole
location coexist with different Euler characteristics. This does not
claim that the two full topological zeta functions coincide; they do not.

\subsection{Denef-Loeser motivic zeta function}

For a germ $(M,f,P)$ on the smooth manifold $M$, Denef-Loeser's motivic zeta  series is
\[
 \mathcal Z_{f,P}^{\hat\mu}(u)
 =\sum_{m\ge1}[\mathcal X_{m,P}(f),\mu_m]\mathbb L^{-mN}u^m.
\]
Its motivic limit is the motivic nearby fibre:
$\mathcal S_{f,P}=-\lim_{u\to\infty}\mathcal Z_{f,P}^{\hat\mu}(u)$.
The jet trace formula supplies a genuine link:
\[
 \chi_c(\mathcal Z_{f,P}^{\hat\mu}(u))
 =\sum_{m\ge1}L(T^m)u^m
 =u\frac{d}{du}\log\zeta_{f,P}(u),
\]
where Euler realization is coefficientwise after forgetting the
finite-group action. Under $T^m$ unipotent the coefficient of $u^m$
is precisely $\chi(F_P)$. This concerns the local contact series, not
the coefficient of $q^m$ in the DT invariants series.  A non-archimedean version of the motivic Donaldson-Thomas series was studied in \cite{Jiang}. 

\subsection{The missing pole information}
The Denef-Loeser motivic zeta function is given by 
\[
 \mathcal Z_{f,P}^{\hat\mu}(u)=
 \sum_{\varnothing\ne I}(\mathbb L-1)^{|I|-1}
 [\widetilde E_I^\circ|_{\pi^{-1}(P)}]
 \prod_{i\in I}
 \frac{\mathbb L^{-\nu_i}u^{N_i}}
      {1-\mathbb L^{-\nu_i}u^{N_i}}.
\]
Here the tilde denotes the standard cyclic cover. The topological
zeta function involves the ratios $-\nu_i/N_i$ and possible
cancellations between resolution strata. The nearby-cycle limit
replaces each displayed fraction by $-1$, removing its explicit
$\nu_i$ dependence. Integration to an absolute virtual motive loses
further information about the location of the stalks.

A simple illustration is provided by $f=xy$ and $g=xy^2$ on $\C^2$.
Both have general fibre of class $\mathbb L-1$ and central fibre of
class $2\mathbb L-1$. The equivariant fibre-difference formula gives
absolute virtual motive $1$ for both critical loci. But at the origin
\[
 Z_{\mathrm{top},xy,0}(s)=\frac1{(s+1)^2},\qquad
 Z_{\mathrm{top},xy^2,0}(s)=\frac1{(s+1)(2s+1)}.
\]
Thus an integrated virtual motive does not in general determine
local topological poles. This example is not a counterexample to the
monodromy conjecture: the $-1$ eigenvalue predicted by the second
pole occurs at nearby points $(a,0)$, $a\ne0$.

\subsection{The higher-power}
A bound $T_P^m$ unipotent says that eigenvalues at $P$ lie in $\mu_m$;
it does not say that every element of $\mu_m$ occurs. Nor does a bound
at one point control monodromy at every nearby point. Consequently,
even a proposed pole-denominator bound requires careful localization.
A valid proof would need both a theorem controlling actual poles and
a nonvanishing statement for their predicted eigenvalues. For the
strong conjecture it would also need to locate exact Bernstein--Sato
roots, rather than their exponentials modulo integral shifts.

\subsection{A positive result}
There is a direct local calculation for $f_n$.

\begin{proposition}\label{prop:smoothzeta}
At a reduced length-$n$ subscheme $P$, with $n\ge2$, put
$r=2n(n-1)$. Then
\begin{equation}\label{eq:smoothzeta}
 Z_{\mathrm{top},f_n,P}(s)
 =\frac{r}{(2s+r)(s+1)}
 =\frac{n(n-1)}{(s+n(n-1))(s+1)}.
\end{equation}
The weak monodromy conjecture holds for this germ, with the predicted
eigenvalues occurring already at $x$.
\end{proposition}
\begin{proof}
The critical scheme is smooth of dimension $3n$ near $P$.
The holomorphic Morse lemma with parameters therefore gives a
nondegenerate quadratic form in $r=N-3n$ transverse variables,
with $3n$ free variables. Smooth factors do not change the local
zeta function or Milnor cohomology.

Blow up the origin in the transverse $\C^r$. The exceptional divisor
has numerical data $(N,\nu)=(2,r)$; the strict transform has data
$(1,1)$. Its intersection with the exceptional divisor is the smooth
projective quadric $Q_{r-2}\subset\PP^{r-1}$. As $r$ is even,
$\chi(Q_{r-2})=r=\chi(\PP^{r-1})$. The resolution formula gives
\[
 \frac{\chi(\PP^{r-1}\setminus Q_{r-2})}{2s+r}
 +\frac{\chi(Q_{r-2})}{(2s+r)(s+1)}
 =\frac{r}{(2s+r)(s+1)}.
\]
Both poles, $-1$ and $-r/2$, are integers. The Milnor fibre retracts
onto $\mathbb{S}^{r-1}$; its monodromy is the antipodal action, acting on
reduced cohomology by $(-1)^r=1$. The eigenvalue predicted by either
pole is $1$, and it occurs in the fibre cohomology.
\end{proof}

For $n=2$ the formula is $2/((s+2)(s+1))$; for $n=3$ it is
$6/((s+6)(s+1))$. In both cases the alternating monodromy zeta function
is $1$, because $H^0$ and $H^{r-1}$ have the same eigenvalue and
opposite Euler signs. Thus a trivial alternating monodromy zeta
function can coexist with nontrivial topological poles.
For $n=1$ the trace potential is identically zero, so the usual
nonconstant-germ topological-zeta question is excluded.
This proposition proves a local case, not the conjecture at all
singular Hilbert points.

\end{document}